\documentclass[12pt]{amsart}
\usepackage{amsfonts,amsmath,amssymb,latexsym}
\usepackage{eucal}
\usepackage{color}
\usepackage{geometry}
\usepackage{amsfonts, amsmath, amssymb,amscd,indentfirst, graphicx}
\usepackage{enumerate}
\usepackage[latin1]{inputenc}

\newtheorem{theorem}{Theorem}
\newtheorem{proposition}[theorem]{Proposition}
\newtheorem{corollary}[theorem]{Corollary}

\newtheorem{definition}{Definition}
\newtheorem{example}{Example}
\newtheorem{remark}{Remark}
\newtheorem*{theorem1*}{Theorem I}
\newtheorem*{theorem2*}{Theorem II}
\usepackage{calrsfs}
\DeclareMathAlphabet{\pazocal}{OMS}{zplm}{m}{n}

\newcommand{\Ga}{\mathcal{G}}

\newcommand{\R}{\mathbb{R}}
\newcommand{\h}{\mathbb{H}}

\begin{document}

\title{Translating solitons in Riemannian products}

\author[J. H. de Lira ]{Jorge H. de Lira}
\address[Lira]{
  Departamento de Matem\'atica,
  Universidade Federal do Cear\'a, Bloco 914, Campus do Pici,
  Fortaleza, Cear\'a, 60455-760, Brazil.
}
\email{jorge.lira@mat.ufc.br}
\author[F. Mart\'in]{Francisco Mart\'\i{}n}
\address[Mart\'in]{
  Departamento de Geometr\'\i{}a y Topolog\'\i{}a,
  Universidad de Granada,
  18071 Granada, Spain.
}
\email{fmartin@ugr.es}

\thanks{J. H. de Lira  is supported by PRONEX/FUNCAP/CNPq PR2-0101-00089.01.00-15 and CNPq/Edital Universal 409689/2016-5.  F. Mart\'in is partially supported by the
  MINECO/FEDER grant MTM2014-52368-P and  by the
Leverhulme Trust grant IN-2016-019. }
\date{\today}
\maketitle
\begin{abstract}
In this paper we study solitons invariant with respect to the flow generated by a complete Killing vector field in a ambient Riemannian manifold. A special case occurs when the ambient manifold  is the Riemannian product $(\mathbb{R} \times P, {\rm d}t^2+g_0)$ and the Killing field is $X=\partial_t$. Similarly to what happens in the Euclidean setting, we call them  {\it translating solitons}. 
We see that a translating soliton in $\mathbb{R} \times P$ can be seen as a minimal submanifold for a weighted volume functional. Moreover we show that this kind of solitons appear in a natural way in the context of a monotonicity formula for the mean curvature flow in $\mathbb{R} \times P$.
When $g_0$ is rotationally invariant and its sectional curvature is non-positive, we are able to characterize all the rotationally invariant translating solitons. Furthermore, we use these families of new examples as barriers to deduce several non-existence results.
\end{abstract}

\section{Introduction}



The study of solitons of the mean curvature flow is intimately associated to the nature of singularities
of this flow. In this paper, we focus on families of solitons invariant with respect to the flow generated by a complete Killing vector field in a ambient Riemannian manifold. Those one parameter families can be seen, up to inner diffeomorphisms, as ancient solutions of the mean curvature flow which evolve by ambient isometries.

A special case occurs when the Riemannian manifold $(\bar M^{n+1},g)$ is the Riemannian product $(\mathbb{R} \times P, {\rm d}t^2+g_0)$ where $(P^n,g_0)$ is a Riemannian  $n$-manifold 
and ${\rm d}t^2$ denotes the Euclidean metric of the real line. In this class of manifolds, the flow generated by the coordinate vector field $X=\partial_t$ is parallel and defines a one-parameter flow of  isometries  we will refer to as {\em vertical translations}. Hence, we say that a submanifold $M$  in $\R \times P$ is a translating soliton of the mean curvature flow if 
\[
c X^{\perp}=\mathbf{H},
\]
where $c\in\mathbb{R}$ is a constant that indicates the velocity of the flow. When $M$ is a hypersurface, the previous equation is equivalent to $H=c   \langle X,N\rangle$, where $N$ means the Gauss map of $M$ and we have denoted $g(\cdot, \cdot)=\langle \cdot,\cdot\rangle$.

Translating solitons have been extensively studied in the particular case $P=\mathbb{R}^n$ (see for instance \cite{Altschuler-Wu, Bourni-Langford, MSHS15, MPGSHS15}). In the Euclidean setting, if the initial immersion has nonnegative mean curvature, then it is known \cite{HS} that any limiting flow of a {\em type-II singularity} has convex time slices hypersurfaces $M_\tau$. Furthermore, either $M_\tau$ is a strictly convex translating soliton or (up to rigid motion) $M_\tau=\R^{n-k} \times \Sigma_\tau^k$, where $\Sigma_\tau^k$ is a lower $k$-dimensional strictly convex translating soliton in $\mathbb{R}^{k+1}$. 

Similarly to what happens in the Euclidean space, a translating soliton $M$ in $\mathbb{R} \times P$ can be seen as a minimal surface for the weighted volume functional 
\[
\mathcal{A}_{c\eta} [M]= \int_M e^{c\eta}\, {\rm d}M
\]
where $\eta$ represents the height Euclidean function, that is, the restriction of the natural coordinate $t\in \mathbb{R}$ to $M$.
Furthermore, in Subsection \ref{subsec:mono} we shall see that the translating soliton equation appears naturally in the context of a monotonicity formula associated to the mean curvature flow in $\mathbb{R}\times P$. 

The last decade has witnessed the construction of many families of new translating solitons in $\R^3$ using different techniques, see \cite{DDPN, hoffman, Nguyen09, Nguyen13, Nguyen15, Smith}.
Clutterbuck, Schn\"urer and Schulze in \cite{CSS} (see also \cite{Altschuler-Wu}) proved that there exists an entire, rotationally symmetric, strictly convex graphical translator in $\R^{n+1}$, for $n\geq2$. This example is known as translating paraboloid or bowl soliton. Moreover, they classified all the translating solitons of revolution, giving a one-parameter family of rotationally invariant cylinders $C_\varepsilon$, where $\varepsilon$ represents the neck-size of the cylinder.  The limit of the family $C_\varepsilon$, as $\varepsilon \to 0$, consists of a double copy of the bowl soliton with a singular point at the origin.

In this paper we proved that a similar family of translating solitons exists when the ambient manifold is a Riemmanian product $\mathbb{R} \times P$ where the {\em Riemannian metric is rotationally invariant} (see Section \ref{sec:equi}) and has {\em non-positive sectional curvature}.

\begin{theorem1*}
Let $P$ be a $n$-dimensional complete Riemannian manifold endowed with a rotationally invariant metric $g_0$ whose sectional curvatures are non-positive. 
Then there exists a one-parameter family of rotationally symmetric translating solitons $M_\varepsilon$, $\varepsilon \in [0, +\infty),$ embedded into the Riemannian product $\bar M = \mathbb{R}\times P$. The translating soliton $M_0$ is an entire graph over $P$ whereas each $M_\varepsilon$, $\varepsilon>0$, is a bi-graph over the exterior of a geodesic ball in $P$ with radius $\varepsilon$.
\end{theorem1*}

In the above situation, $P$ is a Cartan-Hadamard $n$-manifold and it makes sense to speak about the ideal boundary of $P$, that we will denote as $\partial_\infty P$. In this context we can obtain the following result:
\begin{theorem2*}
Let $P$ be a $n$-dimensional complete Riemannian manifold endowed with a rotationally invariant metric $g_0$ whose sectional curvatures are negative. There exists a one-parameter family of translating solitons $M^\infty_\varepsilon$, $\varepsilon \in [0, +\infty),$ embedded in $\bar M = \mathbb{R}\times P$ and foliated by horospheres in parallel hyperplanes $P_t$. The ideal points in each $M_\varepsilon^\infty$ lie in an asymptotic line of the form $\mathbb{R}\times \{x_\infty\}$ with $x_\infty \in \partial_\infty P$. 
\end{theorem2*}
We would like to point out that, in the particular case when $P=\mathbb{H}^n$, these examples were already constructed by E. Kocakusakli, M. A. Lawn and M. Ortega in \cite{ortega1, ortega2}.

In the Euclidean space,  X.-J. Wang  \cite{Wang}  characterized the bowl soliton  as the only convex translating soliton in $\mathbb{R}^{n+1}$ which is an entire graph. Very recently, J. Spruck and L. Xiao \cite{spruck-xiao} have proved that a translating soliton which is graph over the whole $\R^2$ must be convex. One can think that it could be also the case of any Riemannian manifold $\R \times P$ when $P$ is endowed with a rotationally invariant metric. However, in this paper we show examples of entire graphs in $\R \times \h^n$ which are not hypersurfaces of revolution, see Subsection \ref{subsec:grim}. These examples are foliated by submanifolds equidistant to a fixed geodesic in horizontal slices of $\R \times \h^n$. This behaviour mimics the geometry the grim reaper cylinder in $\R^{n+1}$. For this reason, we call these examples {\em entire grim reaper graphs.}

The families of examples given by Theorems I and II can be used as barriers to prove that if $P$ satisfies the hypotheses of the theorems, then there are no complete proper translating solitons in $\R \times P$ with compact (possibly empty) boundary and contained in a horocylinder of $\R\times P$. In particular, we prove that there are no  complete ``cylindrically bounded'' solitons in the same spirit of a similar result about minimal submanifolds by Al\'\i as, Bessa and Dajczer \cite{ABD}.








\section{Translating mean curvature flows}
\label{ssf}

Let $(P^{n}, g_0)$ be a  complete Riemannian
manifold and denote the product $\mathbb{R}\times P^n$ by $\bar M^{n+1}$, endowed with the Riemannian product metric.  Given a $m$-dimensional manifold $M^m$ with $m\le n$ and $\omega_* < 0 <\omega^*$, we consider a
differentiable map
\begin{equation}
\Psi:(\omega_*,\omega^*)\times M\to \bar M
\end{equation}
such that $\Psi_{\tau} = \Psi(\tau,\cdot)$ is an immersion, for all $\tau\in (\omega_*,\omega^*)$. We denote
 $\psi=\Psi_0$. The submanifolds $\Psi_{\tau}(M),\,\tau\in
(\omega_*,\omega^*),$ are evolving by their mean curvature vector
field if
\begin{equation}
\frac{d\Psi}{d\tau}=\Psi_* \frac{\partial}{\partial \tau} = {\bf
H},
\end{equation}
 where
\begin{equation}
{\bf H} =\bigg(\sum_{i=1}^{m} 
\bar\nabla_{\Psi_{\tau*}{\sf e}_i} \Psi_{\tau*} {\sf e}_i\bigg)^\perp
\end{equation}
is the (non-normalized) mean curvature vector of $\Psi_\tau$. Here and in what follows $\perp$ indicates the projection onto the normal bundle; the local  tangent frame $\{{\sf
e}_i\}_{i=1}^m$ is orthonormal with respect to the metric induced in $M$ by $\Psi_\tau$. The notations $\langle\cdot, \cdot\rangle$ and $\bar\nabla$ stand for the Riemannian metric and connection in $\bar M$, respectively.

Denoting by $t$ the natural coordinate in the factor $\mathbb{R}$ in $\bar M$, the coordinate vector field $X=\partial_t$ is a parallel vector field on $\bar M$. We set $\Phi:\mathbb{R}\times \bar M \to \bar M$ to denote the
flow generated by $X$, that is, the map given by
\[
\Phi(s, (t,x)) = (s+t, x), \quad (t,x)\in \bar M, \,\, s\in \mathbb{R}
\]


\begin{definition}
\label{selfsimilar} Let $\bar M^{n+1}$ be the Riemannian product of $\mathbb{R}$ and a Riemannian manifold ${\rm (}P^n, g_0{\rm)}$. Given a $m$-dimensional Riemannian manifold $M^m$ we say that
a mean curvature flow $\Psi:(\omega_*, \omega^*)\times M \to \bar
M$ is \emph{translating} if  there exists an isometric immersion
$\psi:M \to \bar M$ and a reparametrization $\sigma:(\omega_*, \omega^*)\to \mathbb{R}$ such that
\begin{equation}
\Psi_{\tau}(M) =\Phi_{\sigma(\tau)}(\psi(M)),
\end{equation}
for all $\tau\in (\omega_*,\omega^*)$, where
 $\Phi:\mathbb{R}\times \bar M\to\bar M$ is the flow generated by $X$. 
\end{definition}

This condition means that
\begin{equation}
\label{}
\widetilde\Psi_\tau (M) = \psi(M)
\end{equation}
where the map $\widetilde\Psi:(\omega_*, \omega^*)\times M \to \bar M$ is given by
\begin{equation}
\label{psiphi} \Psi(\tau,x) =
\Phi(\sigma(\tau), \widetilde\Psi(\tau,x)), 
\end{equation}
for any $(\tau, x)\in (\omega_*, \omega^*)\times M$.

\begin{example}\label{product-example} {\rm (Translating mean curvature flow of graphs.)}  Fixed $t_0\in \mathbb{R}$, we may define a
mean curvature flow in $\bar M$ by
\[
\Psi(\tau, x) = (c, x), \quad \tau\in \mathbb{R},
\]
for any $x\in P_c=\{c\}\times P$.  Less trivial examples may be given by  translating graphs: consider a function $u: (\omega_*, \omega^*)\times M  \to \mathbb{R}$ and define $\Psi: (\omega_*, \omega^*)\times M \to \mathbb{R}$ as
\[
\Psi(\tau, x) = (u(\tau, x), x).
\]
This defines a mean curvature flow if and only if $u$ satisfies the quasilinear parabolic equation
\begin{equation}
\label{pde-parabolic}
\frac{\partial u}{\partial \tau} =W \, {\rm div} \bigg(\frac{\nabla u}{W}\bigg),
\end{equation}
where 
\[
W = \sqrt{1+|\nabla u|^2}
\]
onde $\nabla$ and ${\rm div}$ are, respectively, the Riemannian connection and divergence in  $(P, g_0)$. This notion of translating soliton has been extensively studied in Euclidean spaces, see for instance {\rm \cite{Altschuler-Wu, CSS, HS, ILMANEN, MSHS15, SHAHRIYARI15, Wang}.} Our definition is the natural setting to these special flows in Riemannian products $\mathbb{R} \times P$.
\end{example}

Next, we present some fundamental consequences of Definition \ref{selfsimilar} that motivates the notion of translating soliton in Riemannian products.

\begin{proposition}
\label{soliton-prop} Let $\Psi:(\omega_*, \omega^*)\times M \to
\bar M$ be a translating mean curvature flow with respect to the parallel vector field $X=\partial_t$. Then for all
$\tau\in (\omega_*, \omega^*)$ there exists a constant $c_\tau$
such that
\begin{equation}
\label{solitonA} c_\tau X =c_\tau\Psi_{\tau *}T + {\bf H},
\end{equation}
where ${\bf H}$ is the mean curvature vector of $\Psi_\tau= \Psi(\tau, \cdot\,)$ and
$T\in\Gamma(TM)$ is the pull-back by $\Psi_\tau$ of the tangential
component of $X$. Furthermore, 
\begin{equation}
\label{solitonC} \bar\nabla^\perp {\bf H}+c_\tau\, II(T,\cdot)=0,
\end{equation}
where $II$ is the second fundamental form of $\Psi_\tau$ and $\bar\nabla^\perp$ is its normal connection.  
Moreover,
\begin{equation}
\label{solitonB} II_{-\mathbf{H}} + \frac{c_\tau}{2}\pounds_T g =0,
\end{equation}
where $g$ is the metric induced in $M$ by $\Psi_\tau$ and $II_{-\bf
H}$ is its second fundamental form in the direction of $-{\bf H}$. 
\end{proposition}

We refer the reader to \cite{ALR} for the proof of this proposition in the more general case of warped product spaces.

Motivated by the above geometric setting, we define a general notion of translating soliton  in Riemannian products as follows.

\begin{definition}\label{soliton-definition}
An isometric immersion $\psi:M^m\to \bar M^{n+1} = \mathbb{R}\times P^n$ is a \emph{translating soliton}  with respect to $X=\partial_t$  if
\begin{equation}
\label{solitonA-2}c\, X^\perp = {\bf H}
\end{equation}
along $\psi$ for some constant $c\in \mathbb{R}$. With a slight abuse of notation, we also say that
the submanifold  $\psi(M)$ itself is the translating
soliton {\rm (}with respect to the vector field $X${\rm )}. If $m=n$, that is, for codimension $1$, the condition
becomes
\begin{equation}
\label{soliton-scalar} H=c\,\langle X, N\rangle,
\end{equation}
where  the mean curvature $H$ with respect to the normal vector field $N$ along $\psi$ is given by
\begin{equation}
\label{HHN}
{\bf H} = HN.
\end{equation}
\end{definition}

We observe that equation (\ref{solitonA-2}) is enough to deduce the following important consequences that we have considered in Proposition \ref{soliton-prop} in the context of a geometric flow.



\begin{proposition}
\label{soliton-prop-2} Let $\psi:M^m\to \bar M^n$ be a translating soliton with respect to   $X=\partial_t$. Then along $\psi$ we have
\begin{equation}
\label{solitonB-2} II_{-\mathbf{H}} + \frac{c}{2}\pounds_{T} g =0,
\end{equation}
where $g$ is the metric induced in $M$ by $\psi$ and $II_{-\bf H}$
is its second fundamental form in the direction of $-{\bf H}$. Here   the vector field $T$ is defined by $\psi_*T = X^\top$. Furthermore
\begin{equation}
\label{solitonC-2} \bar\nabla^\perp {\bf H}+c\, II(T,\cdot)=0,
\end{equation}
where $II$ is the second fundamental form of $\psi$ and $\bar\nabla^\perp$ is its normal connection.
\end{proposition}

\noindent \emph{Proof.}  Using (\ref{solitonA-2}) by a direct computation we have for any tangent vector fields $U, V\in \Gamma(TM)$
\begin{eqnarray*}
& & 0=c\langle \bar\nabla_{\psi_* U} X, \psi_* V\rangle+c\langle \bar\nabla_{\psi_* V} X, \psi_* U\rangle\\
& &\,\, = c\langle \bar\nabla_{\psi_* U} \psi_* T, \psi_* V\rangle+c\langle \bar\nabla_{\psi_* V} \psi_* T, \psi_* U\rangle + \langle \bar\nabla_{\psi_* U} {\bf H}, \psi_* V\rangle+\langle \bar\nabla_{\psi_* V} {\bf H}, \psi_* U\rangle\\
& &\,\, = c\langle \nabla_{U} T, V\rangle + c\langle \nabla_{V} T, U\rangle+ 2 II_{-\bf H} (U,V).
\end{eqnarray*}
Hence,
\[
II_{-\bf H} +\frac{c}{2}\pounds_{T} g = 0.
\]
Now one has
\begin{eqnarray*}
& & 0 = c(\bar\nabla_{\psi_* U} X)^\perp  = c (\bar\nabla_{\psi_*U}\psi_*T)^\perp + c(\bar\nabla_{\psi_* U} X^\perp)^\perp\\
& & \,\, = c\,II (T, U) + (\nabla_{\psi_*U}{\bf H})^\perp = c\, II (T, U) + \nabla_{\psi_*U}^\perp{\bf H}
\end{eqnarray*}
what concludes the proof of Proposition \ref{soliton-prop-2}. \hfill $\square$

\begin{remark}
\label{ricci-mcf} The notion of translating solitons and propositions \ref{soliton-prop} and \ref{soliton-prop-2} are particular cases of the general notions of self-similar mean curvature flow and mean curvature flow soliton as formulated  by L. Al\'\i as, J. Lira and M. Rigoli in {\rm \cite{ALR}}.
\end{remark}

\subsection{Variational setting}

For the sake of simplicity in this section we restrict ourselves to the case of codimension $1$ translating solitons in $\bar M^{n+1} = \mathbb{R}\times P^n$. We set   
\[
\eta(x) = (\pi\circ\psi)(x), \quad x\in M,
\]
where $\pi: \mathbb{R}\times P \to \mathbb{R}$ is the projection $\pi(t,x) = t$ and $\psi: M^n\to \bar M^{n+1}$ is a given isometric immersion.

We introduce, for a fixed $c\in \mathbb{R}$, the weighted volume functional
\begin{equation}
\label{vol-Mw}
\mathcal{A}_{c\eta} [\psi,\Omega]={\rm vol}_{c\eta} (\psi(\Omega)) = \int_\Omega e^{c\eta}\, {\rm d}M, 
\end{equation}
where ${\rm d}M$ is the volume element induced in $M$ from $\psi$  and $\Omega$ is a relatively compact domain of $M$. We have the following 

\begin{proposition}
\label{first-variation}
Let $\psi: M^n \to \bar M^{n+1}= \mathbb{R}\times P$ be a codimension $1$ translating soliton with respect to the parallel vector field $X= \partial_t$. Then the equation 
\begin{equation}
\label{sol-scalar-bis}
H = c\,\langle X, N\rangle
\end{equation}
on the relatively compact domain $\Omega \subset M$ is the Euler-Lagrange equation of the functional {\rm (\ref{vol-Mw})}. Moreover, the second variation formula for normal variations is given by
\begin{equation}
\label{second-variation-vol}
\delta^2\mathcal{A}_{c\eta}[\psi, \Omega]\cdot (f,f)=-\int_M e^{c\eta}\,
f L_X f\, {\rm d}M,\quad f\in C^\infty_0(\Omega),
\end{equation}
where the stability operator $L_X$ is defined by
\begin{equation}
L_X u = \Delta_{-c\eta} u + (|II|^2 +{\rm Ric}_{\bar M}(N, N)) u
\end{equation}
where
\begin{equation}
\label{deltaT-op}
\Delta_{-c\eta} =\Delta + c\langle \nabla \eta, \nabla\, \cdot\,\rangle= \Delta +c \langle X, \nabla\,\cdot\,\rangle.
\end{equation}
\end{proposition}

\noindent \emph{Proof.} Given $\varepsilon >0$, let  $\Psi:(-\varepsilon, \varepsilon)\times M \to \bar M$ be a variation of $\psi$ compactly supported in $\Omega\subset M$  with $\Psi(0,\,\cdot\,) = \psi$ and
normal variational vector field
\[
\frac{\partial\Psi}{\partial s}\Big|_{s=0} = fN+T
\]
for some function $f\in C^\infty_0 (\Omega)$ and a tangent vector field $T\in \Gamma(TM)$. Here, $N$ denotes a local unit normal vector field along $\psi$.  Then
\begin{eqnarray*}
& &\frac{{\rm d}}{{\rm d}s}\Big|_{s=0}{\rm vol}_{c\eta}[\Psi_s(\Omega)] =
\int_\Omega e^{c\eta}\,(c\langle X, N\rangle-H) f\, {\rm d}M + \int_\Omega {\rm div} (e^{c\eta}T)\, {\rm d}M.
\end{eqnarray*}
Hence, stationary immersions for variations fixing  the boundary of $\Omega$ are characterized by the
scalar soliton equation
\[
H-c\,\langle X, N\rangle =0\,\,\, \mbox{ on } \,\,\, \Omega \subset M
\]
which yields (\ref{sol-scalar-bis}). Now we compute the second variation formula. At a stationary immersion
we have
\begin{eqnarray*}
& &\frac{{\rm d}^2}{{\rm d}s^2}\Big|_{s=0}{\rm vol}_{c\eta}[\Psi_s(\Omega)] = 
\int_M e^{c\eta}\,\frac{d}{ds}\Big|_{s=0}(c\langle X_s,
N_s\rangle-H_s) f\, {\rm d}M.
\end{eqnarray*}
Using the fact that
\begin{equation}
\label{der-N}
\bar\nabla_{\partial_s} N = -\nabla f - AT,
\end{equation}
we compute
\begin{equation}
\label{der-angle}
\frac{{\rm d}}{{\rm d}s}\Big|_{s=0} \langle X, N\rangle = \langle
\bar\nabla_{\partial_s}X, N\rangle
+\langle X, \bar\nabla_{\partial_s} N\rangle=\langle X, -\nabla f-AT\rangle.
\end{equation}
Since
\begin{equation}
\label{der-H}
\frac{{\rm d}}{{\rm d}s}\Big|_{s=0} H = \Delta f + (|II|^2+{\rm Ric}_{\bar
M}(N,N))f + \pounds_T H,
\end{equation}
and $\nabla \eta = X^\top$ we obtain for normal variations (when $T=0$)
\begin{eqnarray*}
& & \frac{{\rm d}}{{\rm d}s}\big(H-c\langle X, N\rangle\big) =\Delta f +c\langle
X,  \nabla f\rangle +|II|^2 f + {\rm Ric}_{\bar
M}(N,N)\,f\\
& & \,\, =\Delta_{-c\eta} f + |II|^2 f + {\rm Ric}_{\bar M}(N,N) f.
\end{eqnarray*}
This finishes the proof of the proposition. \hfill $\square$

\vspace{3mm}

\noindent This proposition motivates the definition of the weighted mean curvature as
\[
\widetilde H = H - c\langle X, N\rangle.
\]
Then a mean curvature flow soliton can be considered as a weighted minimal hypersurface.

\begin{remark} Notice that $L_X H =0$.
\end{remark}

In the sequel we need the following particular version of Proposition 5.1 in \cite{ALR}

\begin{theorem}
Let $\psi: M\to \bar M = \mathbb{R}\times P$ be a translating soliton. Then 
\begin{equation}
\Delta_{-c\eta}\eta = c
\end{equation}
\end{theorem}

\noindent \emph{Proof.} Since $\bar\nabla t = X$  we have
\[
\nabla \eta = X - X^\perp 
\]
where $\perp$ denotes the projection onto the normal bundle of $\psi$. Then
\[
\Delta \eta =  - \sum_{i=1}^m \langle \bar\nabla_{{\sf e}_i} X^\perp, {\sf e}_i\rangle = \langle {\bf H}, X^\perp\rangle = c|X^\perp|^2
\]
and
\[
\Delta_{-c\eta} \eta = \Delta \eta + c\langle X, \nabla \eta\rangle = c|X^\perp|^2 + c (|X|^2-|X^\perp|^2) = c|X|^2 =c
\]
This finishes the proof. \hfill  $\square$

\vspace{3mm}

\subsection{A monotonicity formula for translating mean curvature flows.} \label{subsec:mono}

The translating soliton equation (\ref{solitonA-2}) appears naturally in the context of a monotonicity formula associated to the mean curvature flow in $\bar M = \mathbb{R}\times P$. 

\begin{proposition}
The function $K\circ\Psi$ where
\begin{equation}
K(t,\tau)  = \exp(ct-c^2\tau)
\end{equation}
is monotone non-increasing along the mean curvature flow. 
\end{proposition}

\noindent \emph{Proof.} We consider a positive function $\varrho: (\omega_*, \omega^*)\times M \to \mathbb{R}$ of the form
\[
\varrho(\tau, x) = K(\tau, t(\Psi(\tau, x)))
\]
for some function $K: (\omega_*, \omega^*)\times I \to \mathbb{R}$ to be determined. Denoting $M_\tau = \Psi_\tau(M)$ we define
\begin{equation}
F(\tau) = \int_{M} \varrho (\tau, x)\, {\rm d}M_\tau
\end{equation}
where ${\rm d}M_\tau$ is the volume element induced in $M$ by the immersion $\Psi_\tau$.  Considering the mean curvature flow
\[
\frac{\partial\Psi}{\partial\tau} = {\bf H}_\tau = H_\tau N_\tau
\]
one obtains
\begin{eqnarray*}
\frac{{\rm d}F}{{\rm d}\tau} = \int_M \Big(\frac{\partial K}{\partial \tau} + \frac{\partial K}{\partial t} \langle \bar\nabla t, {\bf H}_\tau\rangle-K H_\tau^2 \Big){\rm d}M_\tau.
\end{eqnarray*}
However
\[
\Delta_\tau t|_{M_\tau} = {\rm tr}_{M_\tau}\bar\nabla^2 t +\langle \bar\nabla t, {\bf H}_\tau\rangle
\]
from what follows that
\begin{eqnarray*}
\frac{{\rm d}F}{{\rm d}\tau} = \int_M \Big(\frac{\partial K}{\partial \tau} + \frac{\partial K}{\partial t} \big(\Delta_\tau t -{\rm tr}_{M_\tau}\bar\nabla^2 t \big)-K H_\tau^2 \Big){\rm d}M_\tau.
\end{eqnarray*}
On the other hand
\[
\Delta_\tau K =\frac{\partial K}{\partial t} \Delta_\tau t + \frac{\partial^2 K}{\partial t^2} |\bar\nabla t^T|^2.
\]
Hence we have
\begin{eqnarray*}
\frac{{\rm d}F}{{\rm d}\tau} = \int_M \Big(\frac{\partial K}{\partial \tau} + \Delta_\tau K - \frac{\partial^2 K}{\partial t^2} |\bar\nabla t^T|^2 - \frac{\partial K}{\partial t} {\rm tr}_{M_\tau}\bar\nabla^2 t -K H_\tau^2 \Big){\rm d}M_\tau.
\end{eqnarray*}
Now we sum the (zero) integral of
\[
2\Delta_\tau K = 2\frac{\partial K}{\partial t} \big({\rm tr}_{M_\tau}\bar\nabla^2 t +\langle \bar\nabla t, {\bf H}_\tau\rangle \big) + 2\frac{\partial^2 K}{\partial t^2}|\bar\nabla t^T|^2  
\]
obtaining
\begin{eqnarray*}
\frac{{\rm d}F}{{\rm d}\tau} = \int_M \Big(\frac{\partial K}{\partial \tau} + \Delta_\tau K + \frac{\partial^2 K}{\partial t^2} |\bar\nabla t^T|^2 + \frac{\partial K}{\partial t} {\rm tr}_{M_\tau}\bar\nabla^2 t  + 2\frac{\partial K}{\partial t} \langle \bar\nabla t, {\bf H}_\tau\rangle-K H_\tau^2 \Big){\rm d}M_\tau.
\end{eqnarray*}
Since ${\rm tr}_M \bar\nabla^2 t =0$ we  conclude that
\begin{eqnarray*}
\frac{{\rm d}F}{{\rm d}\tau} = \int_M \Big(\frac{\partial K}{\partial \tau} + \Delta_\tau K + \frac{\partial^2 K}{\partial t^2} |\bar\nabla t^T|^2 + 2\frac{\partial K}{\partial t} \langle \bar\nabla t, {\bf H}_\tau\rangle-K H_\tau^2 \Big){\rm d}M_\tau.
\end{eqnarray*}
Rerranging terms we get
\begin{equation}
\frac{{\rm d}F}{{\rm d}\tau} = \int_M \Big(\frac{\partial K}{\partial \tau} +\frac{\partial^2 K}{\partial t^2} -\frac{\partial^2 K}{\partial t^2} \langle X, N_\tau\rangle^2+ 2\frac{\partial K}{\partial t} H_\tau \langle X, N_\tau\rangle-K H_\tau^2 \Big){\rm d}M_\tau.
\label{monotonicity-0}
\end{equation}
Supposing that 
\begin{equation}
\frac{\partial }{\partial t}\bigg(\frac{1}{K}\frac{\partial K}{\partial t}\bigg) =0
\end{equation}
one has
\begin{equation*}
\frac{{\rm d}F}{{\rm d}\tau} = \int_M \bigg(\frac{\partial K}{\partial \tau} +\frac{1}{K}\Big(\frac{\partial K}{\partial t}\Big)^2 -K\Big( \frac{1}{K}\frac{\partial K}{\partial t}\langle X, N_\tau\rangle- H_\tau\Big)^2 \bigg){\rm d}M_\tau.
\end{equation*}
Hence we set $K(\tau, t)$ as 
\[
K(\tau, t) =\exp(ct-c^2\tau).
\]
With this choice we conclude that
\begin{equation}
\label{monotonicity-4}
\frac{{\rm d}F}{{\rm d}\tau} =- \int_M K\big( H_\tau - c\langle X, N_\tau\rangle \big)^2{\rm d}M_\tau = - \int_M K\big|{\bf H}_\tau - c X^\perp\big|^2{\rm d}M_\tau, 
\end{equation}
what implies that $F$ is non-increasing along the mean curvature flow. \hfill $\square$

%
%

\subsection{Translating soliton equation}

Translating solitons may be locally described in non-parametric terms as graphs of solutions of a quasilinear PDE in divergence form. 

\begin{proposition}
Given a domain $\Omega \subset P$ and a $C^2$ function $u:\Omega \to \mathbb{R}$ the graph
\begin{equation}
M = \{(x, u(x): x\in \Omega\}\subset \bar M
\end{equation}
is a translating soliton if and only if $u$ satisfies the quasilinear partial differential equation
\begin{equation}
\label{PDE-soliton-1}
{\rm div}\bigg(\frac{\nabla u}{\sqrt{1+|\nabla u|^2}}\bigg)  = \frac{c}{\sqrt{1+|\nabla u|^2}}
\end{equation}
for some constant $c\in \mathbb{R}$.
\end{proposition}

\noindent \emph{Proof.} Given a $C^2$ function $u:\Omega \to \mathbb{R}$ defined in an open subset $\Omega \subset P$ we parameterize its graph $M$ by
\[
\psi(x) = (u(x), x)\in \mathbb{R} \times \Omega.
\]
The induced metric has local components of the form
\[
g_{ij} = \sigma_{ij} + u_i u_j
\]
where
\[
\sigma_{ij} = g_0 (\partial_i, \partial_j)
\]
are the local coefficients of the metric $g_0$ in terms of local coordinates in $P$.  We fix the orientation of $M$ given by the unit normal vector field
\begin{equation}
\label{graph-normal}
N = \frac{1}{W} \big(X- \nabla u\big)
\end{equation}
where
\[
W = \sqrt{1+|\nabla u|^2}.
\]
 The second fundamental form of $M$ with respect to  $N$ has local components 
\begin{equation}
\langle \bar\nabla_{\psi_*\partial_i}\psi_*\partial_j, N\rangle = \frac{1}{W} u_{i;j}
\end{equation}
where $u_{i;j}$ are the components of the Hessian of $u$ in $P$. It follows that the mean curvature $H$ of $M$ is given by
\begin{equation}
H =\frac{1}{W}\bigg(\Delta u -\frac{1}{1+|\nabla u|^2} \langle \nabla_{\nabla u}\nabla u, \nabla u\rangle\bigg).
\end{equation}
This equation can be written in divergence form 
as follows
\begin{equation}
\label{PDE-H}
H = {\rm div}\bigg(\frac{\nabla u}{W}\bigg).
\end{equation}
On the other hand the scalar translating soliton equation (with constant $c$) is
\[
H = c\langle X, N\rangle = \frac{c}{W}
\]
from what follows that the translating soliton PDE is 
\begin{equation}
\label{PDE-1}
\Delta u -\frac{1}{1+|\nabla u|^2} \langle \nabla_{\nabla u}\nabla u, \nabla u\rangle = c.
\end{equation}
Therefore we conclude that (\ref{PDE-1}) can be written as
\begin{equation}
\label{PDE-div}
\sqrt{1+|\nabla u|^2}\,{\rm div} \bigg(\frac{\nabla u}{\sqrt{1+|\nabla u|^2}}\bigg) = c.
\end{equation}
This finishes the proof. \hfill $\square$

\vspace{3mm}

Integrating (\ref{PDE-soliton-1})  with respect to the Riemannian measure ${\rm d}P$ in $P$ we obtain
\begin{eqnarray*}
\int_\Omega \frac{c}{W}  \,{\rm d}P=\int_\Omega {\rm div} \Big(\frac{\nabla u}{W}\Big)\, {\rm d}P =\int_{\partial\Omega} \Big\langle \frac{\nabla u}{W}, \nu_\Omega\Big\rangle  {\rm d}\sigma,
\end{eqnarray*}
where ${\rm d}\sigma$ is the Riemannian measure in $\partial\Omega$ and $\nu_\Omega$ is the outward unit conormal along $\partial\Omega$. We then obtain the following expression that can be regarded as an analog of the flux formula in the case of constant mean curvature graphs:
\begin{equation}
\label{flux-solitons}
\int_\Omega \frac{c}{W} \,{\rm d}P +\int_{\partial\Omega} \Big\langle \frac{\nabla u}{W}, \nu_\Omega\Big\rangle\, {\rm d}\sigma=0.
\end{equation}


\section{Equivariant examples of translating solitons} \label{sec:equi}

From now on,  we suppose that  the Riemannian metric $g_0$ in $P$ has non-positive sectional curvatures and it is rotationally symmetric in the sense that it can be expressed as
\[
g_0 = {\rm d}r^2 + \xi^2(r)\, {\rm d}\vartheta^2,
\]
where ${\rm d}\vartheta^2$ stands for the metric in $\mathbb{S}^{n-1}\subset \mathbb{R}^n$ and $\xi$ is an even function on $\mathbb{R}$ with
\begin{equation}
\begin{cases}
& \xi(r)>0, \,\, \mbox{ on } \,\, r>0\\
& \xi'(0)=1,\\
& \xi^{(2k)}(0)=0, \quad k\in \mathbb{N}.
\end{cases}
\end{equation}
Then a radial function $u=u(r)$ is solution of (\ref{PDE-soliton-1}) if and only if it satisfies the ODE
\[
\sqrt{1+u'^2}\bigg(\frac{u'}{\sqrt{1+u'^2}}\bigg)'+ u' \Delta r = c,
\]
that is, 
\begin{equation}
\label{soliton-ode-0}
\frac{u''}{1+u'^2}  + (n-1)\frac{\xi'(r)}{\xi(r)} u' = c
\end{equation}
where $'$ denotes derivatives with respect to the radial coordinate $r$. Here we have used the fact that 
\[
\Delta r = (n-1)\frac{\xi'(r)}{\xi(r)}\cdot
\]
Note that in the case when $u=u(r)$ and $\Omega = B_{r}(o)$, the geodesic ball in $P$ centered at $o$ with radius $r>0$, we have $\nabla u = u'(r)\partial_r$ and  (\ref{flux-solitons}) reduces to 
\begin{equation}
\label{flux-graphs}
\frac{u'(r)}{\sqrt{1+u'^2(r)}} \xi^{n-1}(r) = \int_0^r H\xi^{n-1}(\tau)\, {\rm d}\tau.
\end{equation}
with
\[
H = H(r) =  \frac{c}{\sqrt{1+u'^2(r)}}\cdot
\]
Taking derivatives with respect to $r$ in both sides we get
\begin{eqnarray*}
\bigg(\frac{u''}{W} - \frac{1}{W^3} u'^2 u'' \bigg) \xi^{n-1} + (n-1)\frac{u'}{W}\xi^{n-1}\frac{\xi'}{\xi} = H(r)\xi^{n-1} = \frac{c}{W}\xi^{n-1},
\end{eqnarray*}
that is,
\[
\frac{u''}{1+u'^2}   + (n-1)u'\frac{\xi'}{\xi} =c
\]
Therefore we recover as expected equation (\ref{soliton-ode-0})  from the fact that (\ref{flux-graphs}) is a sort of first integral to the second order ODE (\ref{soliton-ode-0}).

\begin{theorem}
\label{thm-rev-surf}
Let $P$ be a $n$-dimensional complete Riemannian manifold endowed with a rotationally invariant metric $g_0$ whose sectional curvatures are non-positive. Then there exists a one-parameter family of rotationally symmetric translating solitons $M_\varepsilon$, $\varepsilon \in [0, +\infty),$ embedded into the Riemannian product $\bar M = \mathbb{R}\times P$. The translating soliton $M_0$ is an entire graph over $P$ whereas each $M_\varepsilon$, $\varepsilon>0$, is a bi-graph over the exterior of a geodesic ball in $P$ with radius $\varepsilon$.
\end{theorem}

\noindent \emph{Proof.} A rotationally symmetric hypersurface $M$ in $\bar M=\mathbb{R}\times P$ can be parameterized in terms of cylindrical coordinates as
\begin{equation}
\label{curve-param}
(s,\theta) \mapsto (r(s), t(s), \vartheta),
\end{equation}
where $s$ is the arc-lenght parameter of the profile curve in the orbit space defining $M$, that is, the intersection of $M$ and a geodesic half-plane $\Pi^+$ of the form $\vartheta = {\rm constant}$, $r\ge 0$.  Note that
\[
u' = \frac{\dot t}{\dot r}
\]
whenever $\dot r\neq 0$. Moreover if we assume that $\dot r \ge 0$ then
\[
\frac{u'}{W} = \frac{\dot t}{\sqrt{\dot r^2+\dot t^2}}
\]
Let $\phi$ be the angle between the radial coordinate vector field $\partial_r$ and the tangent line to the profile curve.
We claim that $M$ is a rotationally symmetric translating soliton if and only if the functions $(r,t,\phi)$ satisfy the first-order system
\begin{equation}
\label{soliton-ode-system}
\begin{cases}
& \dot r = \cos\phi\\
& \dot t = \sin\phi\\
& \dot \phi = c\cos\phi-(n-1)\frac{\xi'(r)}{\xi(r)} \sin\phi
\end{cases}
\end{equation}
The last equation in the system is deduced as follows
\begin{eqnarray*}
\bigg(\frac{u'}{W}\bigg)' =\ddot t \frac{1}{\dot r} = \cos\phi\, \dot\phi \frac{1}{\cos\phi} = \dot \phi.
\end{eqnarray*}
On the other hand since $W = 1/\dot r=1/\cos\phi$ and $u'/W=\dot t = \sin\phi$ we have
\[
\bigg(\frac{u'}{W}\bigg)' = \frac{c}{W} - \frac{u'}{W} \Delta r = \frac{c}{W} -(n-1)\frac{\xi'(r)}{\xi(r)} \frac{u'}{W} = c\cos\phi-(n-1)\frac{\xi'(r)}{\xi(r)} \sin\phi.
\]
The solutions of this system are complete since it can be regarded as the system of equations of geodesic curves in the Riemannian half-plane $\Pi^+=\{(r,t)\} = \mathbb{R}^+\times \mathbb{R}$ endowed with a metric  $\tilde g = \lambda^2(r,t)({\rm d}r^2+{\rm d}t^2)$ conformal to the Euclidean metric with
\[
\lambda(r,t) = e^{ct}\xi^{n-1}(r).
\]
Indeed  for rotationally symmetric hypersurfaces the weighted area functional (\ref{vol-Mw})  is given by
\begin{eqnarray*}
& & \mathcal{A}_{c\eta}[\Sigma] = \int_{\mathbb{S}^{n-1}}\int_{s_0}^s e^{ct(s)}\xi^{n-1}(r(s))\, {\rm d}s\, {\rm d}\vartheta\\
& & \,\, =|\mathbb{S}^{n-1}|\int_{s_0}^s \lambda(r(s), t(s))\, {\rm d}s\, {\rm d}\vartheta.
\end{eqnarray*}
On the other hand the lenght functional for the  profile  curve  $s\mapsto (r(s), t(s))\in \Pi^+$ computed with respect to the metric $\tilde g$ is given by
\[
L(\gamma) = \int_{s_0}^s \lambda(r(s), t(s))\, {\rm d}s
\]
since $\dot r^2(s) + \dot t^2(s)=1$. The Christoffel symbols for $\tilde g$ are
\[
\Gamma^r_{rr} = \Gamma_{rt}^t = -\Gamma_{tt}^r = \frac{\lambda_r}{\lambda}
\]
and
\[
\Gamma_{tt}^t = \Gamma_{tr}^r = -\Gamma_{rr}^t = \frac{\lambda_t}{\lambda}
\]
what implies that the geodesic system of equations is
\begin{equation*}
\begin{cases}
& \ddot r + \frac{\lambda_r}{\lambda}\dot r^2+ 2\frac{\lambda_t}{\lambda} \dot r \dot t -\frac{\lambda_r}{\lambda} \dot t^2=0\\
& \ddot t - \frac{\lambda_t}{\lambda}\dot r^2+ 2\frac{\lambda_r}{\lambda} \dot r \dot t +\frac{\lambda_t}{\lambda} \dot t^2=0
\end{cases}
\end{equation*}
However
\begin{eqnarray*}
\dot r \ddot t - \dot t \ddot r = (\cos^2\phi+\sin^2\phi)\dot \phi.
\end{eqnarray*}
Therefore
\[
\dot\phi = \frac{\lambda_t}{\lambda}\dot r (\dot r^2+\dot t^2)-\frac{\lambda_r}{\lambda}\dot t(\dot r^2+\dot t^2) = \frac{\lambda_t}{\lambda}\dot r -\frac{\lambda_r}{\lambda}\dot t.
\]
Since
\[
\frac{\lambda_t}{\lambda} = c, \quad \frac{\lambda_r}{\lambda}= (n-1)\frac{\xi'(r)}{\xi(r)}
\]
we conclude that
\[
\dot \phi = c \cos \phi - (n-1)\frac{\xi'(r)}{\xi(r)}\sin\phi.
\]
Therefore it is enough to prove that whenever a geodesic has a limit when $r\to 0^+$ the limit angle is $\phi\to 0$. Those curves can be reflected through the line $r=0$ and then are the geodesics in the complete plane $\Pi$ (defined by the reflection of $\Pi^+$ with respect to $r=0$) with initial conditions $r(0)=0, t(0)=t_0, \phi(0)=0$ and $\dot r(0)=1$, $\dot t(0)=0$ as limits. The other geodesics 
are confined in the half-plane $\Pi^+$. In this case, we have initial conditions of the form $r(0)=\varepsilon, t(0) = t_0, \phi(0) = \pm \frac{\pi}{2}$ with $\dot r(0) = 0$, $\dot t(0) = \pm 1$.

In sum,  rotationally symmetric discs, for instance, correspond to fix initial conditions  as $r(0)=0,\, t(0)= t_0,\, \phi(0)=0$.  Besides these graphs, we have for a fixed $c$ a one-parameter family of bi-graphs given by the solution of the system (\ref{soliton-ode-system}) with initial conditions $r(0)=\varepsilon, \, t(0)=0, \, \phi(0)=\pm \frac{\pi}{2}$. \hfill $\square$

\begin{remark}
\label{names1}
In analogy with the Euclidean case, we refer to the translating solitons $M_0$ and $M_\varepsilon$, $\varepsilon>0$, respectively as \emph{bowl} solitons and \emph{wing-like} solitons.
\end{remark}

\begin{proposition}
\label{u-asymp} Suppose that  there exist negative constants $K_-,$ and $K_+$ such that $K_-\le K\le K_+< 0$ and that $\big(\frac{\xi(r)}{\xi'(r)}\big)'\to 0$ as $r\to+\infty$. 
The rotationally symmetric translating solitons $M_\varepsilon$, $\varepsilon \in [0, +\infty),$ are described, outside a cylinder over a geodesic ball $B_R(o)\subset P$,  as graphs or bi-graphs of functions with the following asymptotic behavior
\begin{equation}
\label{deru-asymp}
u'(r) = \frac{c}{n-1}\frac{\xi(r)}{\xi'(r)} + o\bigg(\frac{\xi'(r)}{\xi(r)}\bigg)
\end{equation}
as $r\to +\infty$. 
\end{proposition}

\noindent \emph{Proof.} In what follows we keep using the notations fixed just above. Whenever $\dot r >0$ we have $W=1/\dot r$ and
\[
\dot W = \frac{1}{\cos^2\phi}\sin\phi\, \dot \phi.
\]
At a maximum point of $W$, that is, in a point where $\dot r>0$, either $\sin\phi =0$ or $\dot \phi=0$. Suppose that $\sin\phi=0$ and then $\cos\phi =1$. Since this happens at a maximum point of $W$ we have
\[
W \le 1,
\]
what implies that $\dot r\ge 1$ and therefore $\dot r\equiv 1$ whenever $\dot r>0$. If it is the case, then $\dot t \equiv 0$ whenever $\dot r>0$. This corresponds to a horizontal plane that is not a translating soliton. The other possibility is that $\dot\phi=0$ (and $\sin\phi\neq 0$) at a maximum point of $W$, say, $r=r_0$.  In this case
\[
c\cos\phi = (n-1)\frac{\xi'(r_0)}{\xi(r_0)}\sin\phi
\]
and since $\xi'>0$ and $c>0$
\[
\dot r^2 \ge \frac{\frac{(n-1)^2}{c^2}\frac{\xi'^2(r_0)}{\xi^2(r_0)}}{1+\frac{(n-1)^2}{c^2}\frac{\xi'^2(r_0)}{\xi^2(r_0)}}>0
\]
what contradicts the fact that either $\dot r\to 0$ at $r\to\varepsilon$ when the curve does not reach the rotation axis or $\dot r\searrow 1$ as $r\to 0$ in the case when the curve reaches the axis. In this case, the only possibility is that $r_0=0$ since $\xi'(r)/\xi(r)\to +\infty$ as $r\to 0$).

We conclude that $W$ has no interior maximum point in the region where $\dot r>0$.  On the other hand it is obvious that $W$ has a minimum point (at $r=R$, say) where $\dot r=1$. Hence, $W$ is non-decreasing in the region when $\dot r>0$ for $r>R$ what implies that
\[
\dot\phi = c\cos\phi - (n-1)\frac{\xi'(r)}{\xi(r)}\sin\phi\ge 0
\]
for $r\ge R$. 

Now, proceeding as in \cite{CSS} we denote $u'=\varphi$ and then  (\ref{soliton-ode-0}) becomes
\[
\varphi' = (1+\varphi^2)\bigg(c-(n-1)\frac{\xi'(r)}{\xi(r)} \varphi\bigg)=: F(r,\varphi(r)).
\]
Given $\epsilon>0$ denote
\[
\zeta(r) = (1-\epsilon) \frac{c}{n-1}\frac{\xi(r)}{\xi'(r)}\cdot
\]
We can prove that  for every given $\epsilon>0$ and $r_0>R$ there exists $r_1>r_0$ such that
\[
\zeta(r_1) \le \varphi(r_1).
\]
If it is not the case then there exist $\epsilon>0$ and $r_0>R$ such that 
\[
\varphi(r) < (1-\epsilon) \frac{c}{n-1} \frac{\xi(r)}{\xi'(r)}
\]
for every $r>r_0$. In this case we would have
\[
\varphi'(r) > c\epsilon(1+\varphi^2(r)), \quad r>r_0 
\]
what implies that
\[
c\epsilon (r-r_*) < \arctan \varphi(r)-\arctan\varphi(r_*), \quad r, r_* > r_0
\]
what contradicts the fact that the solution is complete and then $r\to +\infty$.

Moreover we can prove that
\[
\zeta'(r)\le F(r,\zeta(r)).
\]
for sufficiently large $r>R$. Indeed we have
\begin{eqnarray*}
& & F(r,\zeta(r)) =c\epsilon\bigg(1+ (1-\epsilon)^2 \frac{c^2}{(n-1)^2}\bigg(\frac{\xi(r)}{\xi'(r)}\bigg)^2\bigg)\ge (1-\epsilon) \frac{c}{n-1}\bigg(\frac{\xi(r)}{\xi'(r)}\bigg)'
\end{eqnarray*}
if and only if
\begin{equation}
\label{ineq-1}
(n-1)^2 + (1-\epsilon)^2 c^2 \bigg(\frac{\xi(r)}{\xi'(r)}\bigg)^2 \ge\bigg(\frac{1}{\epsilon}-1\bigg) (n-1)\bigg(\frac{\xi(r)}{\xi'(r)}\bigg)'
\end{equation}
Denoting by $K$ the radial sectional curvatures in $P$ along geodesics issuing from $o$ and  setting $g=\frac{\xi}{\xi'}$ it follows from Riccati's equation that 
\[
g' = 1+Kg^2.
\]
Hence the inequality (\ref{ineq-1}) above becomes
\[
(n-1)^2 + (1-\epsilon)^2 c^2 \bigg(\frac{\xi(r)}{\xi'(r)}\bigg)^2 \ge \bigg(\frac{1}{\epsilon}-1\bigg) (n-1)\bigg(1+K\bigg(\frac{\xi(r)}{\xi'(r)}\bigg)^2\bigg)
\]
or
\[
\bigg(\frac{\xi(r)}{\xi'(r)}\bigg)^2 \ge \frac{\big(\frac{1}{\epsilon}-1\big) (n-1)-(n-1)^2}{(1-\epsilon)^2 c^2-K\big(\frac{1}{\epsilon}-1\big) (n-1)}
\]
for a sufficiently large $r$. Here, we have need of the assumption that  $K\le 0$. For instance, in the hyperbolic space with constant sectional curvature $-K$ it holds that 
\[
\bigg(\frac{\xi(r)}{\xi'(r)}\bigg)^2 = \frac{1}{-K} \frac{\sinh^2 (\sqrt{-K} r)}{\cosh^2(\sqrt{-K}r)} <\frac{1}{-K}\cdot
\]
Then adjusting $r_1>r_0$ to be sufficiently large we conclude from a standard comparison argument for nonlinear ODEs that
\[
\zeta(r)\le \varphi(r)
\]
for every $r>r_0>R$ sufficiently large. We conclude that for every given $\epsilon$ and $r_0>R$
\begin{equation}
\label{ineq-phi}
(1-\epsilon)\frac{c}{n-1}\frac{\xi(r)}{\xi'(r)} \le \varphi(r) \le \frac{c}{n-1} \frac{\xi(r)}{\xi'(r)}
\end{equation}
for sufficiently large $r>r_0$. We set 
\begin{equation}
\varphi(r) = \frac{c}{n-1} \frac{\xi(r)}{\xi'(r)} +\psi(r).
\end{equation}
Note that
\begin{equation*}
\psi'(r) = -(n-1)\psi(r) \frac{\xi'(r)}{\xi(r)}\bigg(1+\bigg(\frac{c}{n-1}\frac{\xi(r)}{\xi'(r)}+\psi(r)\bigg)^2\bigg)-\frac{c}{n-1}\bigg(1+K\bigg(\frac{\xi(r)}{\xi'(r)}\bigg)^2\bigg)
\end{equation*}
It follows from (\ref{ineq-phi}) that $\psi(r) \le 0$ for sufficiently large $r>r_0$. We claim that 
\[
\lim_{r\to+\infty}\psi(r) = 0.
\]
Suppose that given an arbitrary $r_1 \gg 0$ there exists $\epsilon>0$ such that 
\[
\psi(r_*) \le -\epsilon
\]
for some $r_* >r_1$. Otherwise we are done, that is, we would have $\psi(r) \to 0^-$ as $r\to+\infty$. We have from (\ref{ineq-phi}) that
\[
-\psi(r) = |\psi(r)|\le \frac{c}{2(n-1)}\frac{\xi(r)}{\xi'(r)}
\]
for $r>r_0$ sufficiently large.  Since $K_-\le K\le K_+\le 0$ the Hessian comparison theorem \cite{prs} implies that 
 \[
\sqrt{-K_+}\coth (\sqrt{-K_+}r)= \frac{\xi'_+(r)}{\xi_+(r)}\le \frac{\xi'(r)}{\xi(r)} \le \frac{\xi'_-(r)}{\xi_-(r)} = \sqrt{-K_-}\coth (\sqrt{-K_-}r)
 \]
when  $K_+<0$ and 
\[
\frac{1}{r}= \frac{\xi'_+(r)}{\xi_+(r)}\le \frac{\xi'(r)}{\xi(r)} \le \frac{\xi'_-(r)}{\xi_-(r)} = \sqrt{-K_-}\coth (\sqrt{-K_-}r)
 \]
when $K_+=0$. Suppose that $K_+<0$ (the case $K_+=0$ can be handled with as in \cite{CSS}.) Hence
\begin{eqnarray*}
& & \psi'(r_*) \ge (n-1)\epsilon \frac{\xi'_+(r_*)}{\xi_+(r_*)}\bigg(1+\frac{c^2}{4(n-1)^2}\bigg(\frac{\xi_-(r_*)}{\xi'_-(r_*)}\bigg)^2\bigg)-\frac{c}{n-1}\bigg(1+K_+\bigg(\frac{\xi_-(r_*)}{\xi'_-(r_*)}\bigg)^2\bigg)\\
& &\,\,\ge (n-1)\epsilon \frac{\xi'_+(r_*)}{\xi_+(r_*)}\bigg(1-\frac{c^2(1-\delta)^2}{4(n-1)^2}\frac{1}{K_-}\bigg)-\frac{c}{n-1}\bigg(1-(1-\delta)^2\frac{K_+}{K_-} \bigg)\\
& &\,\,\ge (n-1)\epsilon \sqrt{-K_+}\bigg(1-\frac{c^2(1-\delta)^2}{4(n-1)^2}\frac{1}{K_-}\bigg)-\frac{c}{n-1}\bigg(1-(1-\delta)^2\frac{K_+}{K_-} \bigg)\\
& &\,\,\ge  (n-1)\epsilon \sqrt{-K_+}-\frac{c}{n-1}\epsilon^2 >  \frac{(n-1)}{2}\epsilon \sqrt{-K_+} =:\epsilon'
\end{eqnarray*}
 for sufficiently large $r_*>r_2>0$ and
 \[
 \delta = 1 - \sqrt{\frac{K_-}{K_+} (1-\epsilon^2)}
\]
with
\[
\epsilon < \frac{(n-1)^2}{2c}\sqrt{-K_+}.
\]
It follows that for all $r \gg r_*$ we have
\[
\psi(r) \ge \psi(r_*) + \frac{\epsilon'}{2} (r-r_*) \ge -\epsilon
\]
Since $\epsilon>0$ is arbitrary this proves our claim that $\psi(r) \to 0^-$ as $r\to+\infty$. Now we set
\begin{equation}
\lambda(r) = \frac{\xi(r)}{\xi'(r)}\psi(r).
\end{equation}
We claim that  $\lambda(r) \to 0$ as $r\to +\infty$ when $K_+<0$. Since $\psi(r)\to 0^-$ as $r\to+\infty$ we have for an arbitrarily fixed $\epsilon>0$ that
\begin{equation}
\label{mu}
0\le -\frac{\xi'(r)}{\xi(r)}\lambda(r)=-\psi(r)\le \epsilon 
\end{equation}
for sufficiently large $r>0$.  Since
\begin{eqnarray*}
& & \lambda'(r) = -(n-1)\frac{\xi'(r)}{\xi(r)}\lambda(r) \bigg(1+\bigg(\frac{c}{n-1}\frac{\xi(r)}{\xi'(r)}+\frac{\xi'(r)}{\xi(r)}\lambda(r)\bigg)^2\bigg)\\
& &\,\,\,\,+\bigg(1+K\bigg(\frac{\xi(r)}{\xi'(r)}\bigg)^2\bigg)\bigg(\frac{\xi'(r)}{\xi(r)}\lambda(r)-\frac{c}{n-1}\frac{\xi(r)}{\xi'(r)}\bigg).
\end{eqnarray*}
we have
\begin{eqnarray*}
& & \lambda'(r) =   -\frac{c}{n-1}\frac{\xi(r)}{\xi'(r)}\bigg(1+K\bigg(\frac{\xi(r)}{\xi'(r)}\bigg)^2+c\frac{\xi(r)}{\xi'(r)}\psi(r)+2(n-1)\psi^2(r)\bigg)      \\
& &\,\, -(n-1) \psi(r) + \bigg(1+K\bigg(\frac{\xi(r)}{\xi'(r)}\bigg)^2\bigg)\psi(r) - (n-1)\psi^3(r),
\end{eqnarray*}
from what follows that
\begin{eqnarray*}
& & \lambda'(r) \le    -\frac{c}{n-1}\frac{\xi(r)}{\xi'(r)}\bigg(1+K\bigg(\frac{\xi(r)}{\xi'(r)}\bigg)^2+c\lambda(r)+2(n-1)\psi^2(r)\bigg)      \\
& &\,\, +(n-2) \epsilon - K\bigg(\frac{\xi(r)}{\xi'(r)}\bigg)^2\epsilon + (n-1)\epsilon^3.
\end{eqnarray*}
Therefore
\begin{eqnarray*}
\lambda'(r) \le    -\frac{c}{n-1}\frac{\xi(r)}{\xi'(r)}\bigg(1+K\bigg(\frac{\xi(r)}{\xi'(r)}\bigg)^2+c\lambda(r)\bigg)  +(n-2) \epsilon + \frac{K_-}{K_+}\epsilon + (n-1)\epsilon^3.
\end{eqnarray*}
Suppose that $\lambda(r) \ge \delta$ for some $\delta>0$ and $r> r_3$ where $r_3$ is sufficiently large. Since by assumption $\big(\frac{\xi}{\xi'}\big)' = 1+K\frac{\xi^2}{\xi'^2}\to 0$ as $r\to +\infty$ one has
\[
\lambda'(r) \le  -\frac{c}{n-1}\frac{\xi(r)}{\xi'(r)}\frac{\delta}{2} +O(\epsilon) \le -C
\]
for some $C>0$. We conclude that $\lambda(r) < \delta$ for $r>r_3$ sufficiently large. Similarly if we assume that $\lambda(r) \le -\delta$ for sufficiently large $r>0$ we have 
\begin{eqnarray*}
& & \lambda'(r) \ge   -\frac{c}{n-1}\frac{\xi(r)}{\xi'(r)}\bigg(1+K\bigg(\frac{\xi(r)}{\xi'(r)}\bigg)^2-c\delta+2(n-1)\epsilon^2\bigg)      \\
& &\,\, +(n-2) \frac{\xi'(r)}{\xi(r)} \delta - K\bigg(\frac{\xi(r)}{\xi'(r)}\bigg)^2\frac{\xi'(r)}{\xi(r)} \delta + (n-1)\bigg(\frac{\xi'(r)}{\xi(r)}\bigg)^3\delta^3 \ge C
\end{eqnarray*}
for some $C>0$. Thus $\lambda(r) >\delta$ for $r>r_3$ sufficiently large. We conclude that $\lambda(r) \to 0$ as $r\to+\infty$ when $K_+<0$ as claimed. Hence
\begin{eqnarray*}
\varphi(r) = \frac{c}{n-1}\frac{\xi(r)}{\xi'(r)} + o\bigg(\frac{\xi'(r)}{\xi(r)}\bigg)
\end{eqnarray*}
as $r\to+\infty$ when $K_+<0$.  This finishes the proof \hfill $\square$

\begin{proposition} Suppose that $K<0$ and  $\big(\frac{\xi'(r)}{\xi(r)}\big)'\le 0$. Given a rotationally symmetric translating soliton $M_\varepsilon$, $\varepsilon>0$, we have
\begin{equation}
\lim_{\varepsilon\to 0} (t(\varepsilon)-t(r_0))=0,
\end{equation}
where $r_0>0$ is the radius for which $M_\varepsilon$ attains its minumum height, that is, where $\langle X, N\rangle =1$. Moreover
\begin{equation}
\lim_{\varepsilon\to +\infty} (t(\varepsilon)-t(r_0)) <+\infty.
\end{equation}
\end{proposition}

\noindent \emph{Proof.}  Using  (\ref{curve-param}) and  given a small $\delta>0$ one obtains from  (\ref{flux-solitons})  applied to the region $\Omega =B_{r_0}(o)\backslash B_{\varepsilon+\delta}(o)$   that
 \begin{eqnarray*}
& & |\mathbb{S}^{n-1}|\int_{\varepsilon+\delta}^{r_0} \frac{c}{\sqrt{1+u'^2(r)}}\xi^{n-1}(r)\, {\rm d}r =\int_{\partial B_{r_0}(o)}\bigg\langle\frac{\nabla u}{W}, \partial_r\bigg\rangle -\int_{\partial B_{\varepsilon+\delta}(o)}\bigg\langle\frac{\nabla u}{W}, \partial_r\bigg\rangle\\
& &\,\, = |\mathbb{S}^{n-1}| \bigg(\frac{u'(r_0)}{\sqrt{1+u'^2(r_0)}} \xi^{n-1}(r_0)-\frac{u'(\varepsilon+\delta)}{\sqrt{1+u'^2(\varepsilon+\delta)}} \xi^{n-1}(\varepsilon+\delta)\bigg).
\end{eqnarray*}
In the parametric setting, if $r_0= r(s_0)$ and $\varepsilon = r(0)$ we get after passing to the limit as $\delta \to 0^+$
\begin{equation}
\dot t(s_0)\xi^{n-1}(r(s_0))-\dot t(0)\xi^{n-1}(\varepsilon) = \lim_{\delta\to 0^+}\int_{\varepsilon+\delta}^{r_0} \frac{c}{\sqrt{1+u'^2(r)}}\xi^{n-1}(r)\, {\rm d}r. 
\end{equation}
In particular, 
fix $r_0=r(s_0)$ such that $\dot t(s_0)=0$ and $\dot r (s_0)=1$ with $\phi(s_0)=0$. Since $r(0)= \varepsilon$ then $\dot r(0) =0$ and $\dot t(0) =-1$ (with the choice $\phi(0)=-\frac{\pi}{2}$) what implies that
\begin{equation}
\xi^{n-1}(\varepsilon) = \int_{\varepsilon}^{r_0} \frac{c}{\sqrt{1+u'^2(r)}}\xi^{n-1}(r)\, {\rm d}\tau. 
\end{equation}
\begin{figure}[htbp]
\begin{center}
\includegraphics[height=.39\textheight]{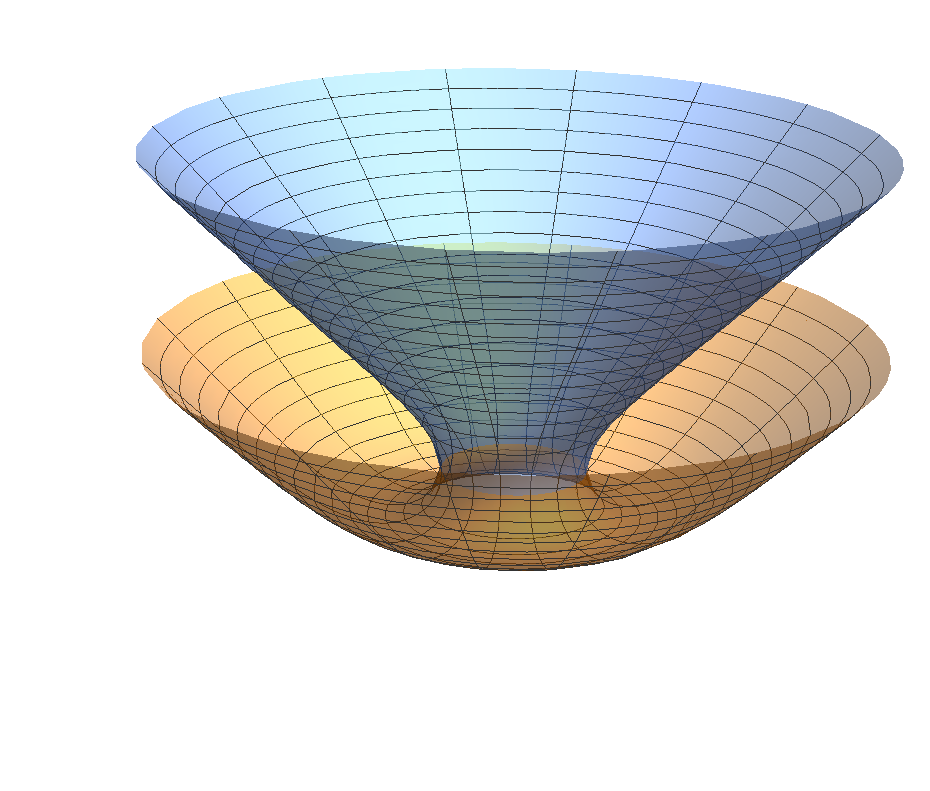} 
\caption{A wing-like solution. It is the union of two graphs over the exterior of a disk (blue and yellow.)}
\label{fig:uno}
\end{center}
\end{figure}

We also have in the region between $r=\varepsilon+\delta$ and $r=r_0$ that
\[
\phi'(r) = \frac{\dot\phi}{\dot r} = \frac{1}{\cos\phi} \dot \phi = c - (n-1)\frac{\xi(r)}{\xi(r)}\tan \phi = c - (n-1)\frac{\xi(r)}{\xi(r)} t'(r).
\]
where $t'(r) = u'(r)$. Therefore 
\begin{equation}
t'(r) = \frac{1}{n-1}\frac{\xi(r)}{\xi'(r)} (c-\phi'(r))
\end{equation}
from what follows since $\dot t\le 0$ in that range that
\[
t(\varepsilon)-t(r_0) = \frac{1}{n-1}\int_{\varepsilon}^{r_0} \frac{\xi(\tau)}{\xi'(\tau)} (\phi'(\tau)-c)\,{\rm d}\tau.
\]
Therefore since $r\mapsto \frac{\xi(r)}{\xi'(r)}$ is non-decreasing 
\begin{eqnarray*}
& & \frac{1}{n-1} \frac{\xi(\varepsilon)}{\xi'(\varepsilon)}  \bigg(\frac{\pi}{2}-c(r_0-\varepsilon)\bigg)=\frac{1}{n-1} \frac{\xi(\varepsilon)}{\xi'(\varepsilon)}  (\phi(r_0)-\phi(\varepsilon)-c(r_0-\varepsilon))\le t(\varepsilon)-t(r_0) \\
& & \,\, \le \frac{1}{n-1} \frac{\xi(r_0)}{\xi'(r_0)}  (\phi(r_0)-\phi(\varepsilon)-c(r_0-\varepsilon))  = \frac{1}{n-1} \frac{\xi(r_0)}{\xi'(r_0)}  \bigg(\frac{\pi}{2}-c(r_0-\varepsilon)\bigg)
\end{eqnarray*}
We conclude that
\begin{equation}
\label{limits}
\frac{1}{n-1} \frac{\xi(\varepsilon)}{\xi'(\varepsilon)}  \bigg(\frac{\pi}{2}-c(r_0-\varepsilon)\bigg) \le t(\varepsilon)-t(r_0) \le  \frac{1}{n-1} \frac{\xi(r_0)}{\xi'(r_0)} \bigg(\frac{\pi}{2}-c(r_0-\varepsilon)\bigg).
\end{equation}
Since $t(\varepsilon) \ge t(r_0)$ this yields
\begin{equation}
\label{rbound}
r_0-\varepsilon \le \frac{\pi}{2}\frac{1}{c}
\end{equation}
and since $r_0\to 0$ as $\varepsilon \to 0$ we have
\begin{equation}
\frac{1}{n-1} \lim_{\varepsilon\to 0}\frac{\xi(\varepsilon)}{\xi'(\varepsilon)}  \frac{\pi}{2} \le \lim_{\varepsilon\to 0} \big(t(\varepsilon)-t(r_0)\big) \le  \frac{1}{n-1} \lim_{r_0\to 0} \frac{\xi(r_0)}{\xi'(r_0)} \frac{\pi}{2}.
\end{equation}
Since 
\[
\lim_{r\to 0} \frac{\xi(r)}{\xi'(r)} = 0
\]
we conclude that 
\[
\lim_{\varepsilon\to 0} \big(t(\varepsilon)-t(r_0)\big) = 0
\]
what is expected from the continuous dependence of the ODE system (\ref{soliton-ode-system}) with respect to the initial conditions.  It also follows from (\ref{limits}) and (\ref{rbound}) that
\[
\lim_{\varepsilon \to+\infty} \big(t(\varepsilon)-t(r_0)\big) <+\infty
\]
This concludes the proof.
\hfill $\square$

As we will see in the last section, there are several interesting applications of the examples constyructed in the previous theorem as barriers for the maximum principle application. Perhaps the simplest one is the following:

\begin{proposition}
In the hypotheses of Theorem \ref{thm-rev-surf}, there are no complete translating solitons in $\R \times P$ contained in a region of the form $\{(s,p) \in \R \times P \; : \; s \leq s_0\},$ for a given $s_0 \in \R$.
\end{proposition}
\begin{proof}
We proceed by contradiction. Assume that there exists such a soliton $M$ in $\R \times P$. We consider the bowl soliton $M_0$. We set $\Phi_s:\mathbb{R}\times P  \to \mathbb{R}\times P$ to denote the
translation
\[
\Phi_s(t,x) = (s+t, x), \quad \mathbb{R}\times P, \,\, s\in \mathbb{R}
\]
As the translations $\{\Phi_s(M_0), \; s \in \R\}$ foliate the whole manifold $\R \times P$, then there exists a first contact between the soliton $M$ and $\Phi_{s_1}(M_0)$, for a suitable $s_1\in\R.$ Taking into account the hypothesis about $M$ and the asymptotic behaviour of $M_0$, the above contact cannot occur at infinity. This means that booth solitons have an interior point of contact and by the maximum principle they must coincide. However, this is impossible the Euclidean coordinate of a bowl soliton is never bounded from above. This contradiction proves the proposition.
\end{proof}

\subsection{Translating solitons with ideal points} \label{subsec:grim}

In this section, we assume that $K<0$. Let $r$ be the distance to a fixed level of a Busemann function in $P$, that is, the distance from a fixed horosphere $\Pi\subset P$ with an ideal point $o$ in the asymptotic boundary $\partial_\infty P$. Then we consider the following coordinate expression of the metric in $P$
\[
{\rm d}r^2 + \xi^2(r) \,{\rm d}\vartheta^2,
\] 
where this time ${\rm d}\vartheta^2$ denotes the Riemannian induced metric in $\Pi$. We then consider translating solitons given as graphs of functions of the form
\[
u = u(r).
\]
Since in this case $\nabla u = u'(r)\nabla r$ and $|\nabla r|=1$ it follows that (\ref{PDE-div}) reduces to 
\begin{equation}
\label{ode-horosphere}
\frac{u''(r)}{1+u'^2(r)} +  u'(r) \Delta r = c.
\end{equation}
In this setting, $\Delta r$ is  the mean curvature of a horosphere, that is, 
\[
\Delta r = (n-1)\frac{\xi'(r)}{\xi(r)}\cdot
\]
We conclude that translating solitons foliated by horospheres are given by solutions of the equation
\begin{equation}
u''(r) =\bigg(c-(n-1)\frac{\xi'(r)}{\xi(r)}\bigg) (1+u'^2(r)).
\end{equation}
We then obtain the following analog of Theorem \ref{thm-rev-surf}

\begin{theorem}
\label{thm-rev-surf-ideal}
Let $P$ be a $n$-dimensional complete Riemannian manifold endowed with a rotationally invariant metric $g_0$ whose sectional curvatures are negative. There exists a one-parameter family of translating solitons $M^\infty_\varepsilon$, $\varepsilon \in [0, +\infty),$ embedded in $\bar M = \mathbb{R}\times P$ and foliated by horospheres in parallel hyperplanes $P_t$. The ideal points in each $M_\varepsilon^\infty$ lie in an asymptotic line of the form $\mathbb{R}\times \{x_\infty\}$ with $x_\infty \in \partial_\infty P$. 
\end{theorem}

\begin{remark}
\label{names2}
We refer to the translating solitons $M_0^\infty$ and $M_\varepsilon^\infty$, $\varepsilon>0$, respectively as \emph{ideal} bowl solitons and \emph{ideal}  wing-like solitons.
\end{remark}

\subsection{Entire grim reaper graphs}

Suppose now that the metric in $P$ may be written  in the form
\begin{equation}
\label{warped-P}
\xi^2(r) \,{\rm d}\tau^2+{\rm d}r^2+ \chi^2(r)\,{\rm d}\vartheta^2
\end{equation}
for $r\in\mathbb{R}$, $\tau \in  \mathbb{R}$ and $\vartheta\in \mathbb{S}^{n-2}$. This metric is regular if $\xi$ satisfies the conditions
\begin{equation}
\begin{cases}
& \xi(0)=1,\\
& \xi^{(2k+1)}(0)=0,\quad k\in\mathbb{N}\\
& \xi(r)>0, \,\, \mbox{ on }\,\, r>0.
\end{cases}
\end{equation}
In this setting, the coordinate vector field $\partial_\tau$ is Killing and the coordinate lines $\tau = {\rm cte.}$ are geodesics.  For $n=2$ the coordinate lines $r = {\rm cte.}$ are equidistant to the geodesic $\varrho=0$.

From now on we suppose that $u= u(r)$. This means that we are searching for special solutions $u$ that depend on only one parameter, mimicking the case of grim reaper cylinders in the Euclidean space. Under this assumption, the translating soliton equation (\ref{PDE-div}) becomes
\[
{\rm div}\bigg(\frac{u'}{\sqrt{1+u'^2}}\nabla r\bigg) = \frac{c}{\sqrt{1+u'^2}}\cdot
\]
Expanding the left-hand side one gets
\begin{eqnarray*}
{\rm div}\bigg(\frac{u'}{\sqrt{1+u'^2}}\nabla r\bigg) =\bigg(\frac{u'}{\sqrt{1+u'^2}}\bigg)' + \frac{u'}{\sqrt{1+u'^2}}
\Delta r
\end{eqnarray*}
In geometric terms, $\Delta r$ is the mean curvature $h$ of the equidistant level sets $r= {\rm cte}$. Note that
\[
(n-1) h = \Delta r= \frac{\xi_r}{\xi}+(n-2)\frac{\chi_r}{\chi}
\]
For $n=2$ we have in particular
\[
h = \frac{\xi_r}{\xi}\cdot
\]
For instance if $\bar M =\mathbb{R} \times \mathbb{H}^2$ with $\xi(r) = \cosh r$  we have 
\[
h =\frac{\xi'(r)}{\xi(r)} = \tanh r.
\]
We conclude that the ODE for translating solitons foliated by equidistant lines is 
\begin{equation}
\label{ode-grim}
u'' = (1+u'^2)(c - (n-1)h(r) u').
\end{equation}
For $P=\mathbb{H}^n$ we have
\begin{equation}
\label{ode-grim-hypn-bis}
u''(r)= (1+u'^2(r))(c - (\tanh r+(n-2)\coth r) u'(r)).
\end{equation}
From (\ref{ode-grim}) it is trivial that $u'=:\varphi$ cannot diverge at a point $x_0\in P$. Indeed, if $\varphi\to \pm\infty$ as $r\to r(x_0)$ then $\varphi'\to \mp \infty$ as  $r\to r(x_0)$, what leads to a contradiction. This fact implies that, given initial values $u(r_0)=t_0$ and $u'(r_0)=\varphi_0$, there exists a unique solution of \eqref{ode-grim} defined for every $r\in\R$. In this way, we have the following existence result

\begin{theorem}
\label{thm-grim}
Let $P$ be a $n$-dimensional complete Riemannian manifold endowed with a metric $g_0$ of the form {\rm(\ref{warped-P})} whose sectional curvatures are negative. Suppose that  $h>0$ and $\lim_{r\to+\infty} h(r)<+\infty$. Then for each $c\in \mathbb{R}$, there exists a translating soliton $M^\infty$ which is an entire graph over $P$ and it is foliated by equidistant lines in parallel hyperplanes $P_t$.
\end{theorem}

\subsection{Equivariant families of examples}

In this section, we summarize the existence results obtained above.

\begin{theorem} \label{th:uno}
Let $M$ be a translating soliton  in $\bar M = \mathbb{R}\times P$ and $\mathcal{G}$ a continuous subgroup
of the isometries of $\bar M$ satisfying  
$g(P)=P$ and $g(M)=M$, for all $g \in \mathcal{G}$. Then we have:
\begin{enumerate}[(a)]
\item[{\rm i.}] If $\Ga$ consists of rotations around a vertical axis, then $M$ is 
part of either a bowl soliton or a wing-like soliton.
\item[{\rm ii.}]  If $\Ga$ consists of hyperbolic translations along a fixed geodesic $\gamma$ in $P$, then
$M$ is an open region of the grim reaper hyperplane.
\item[{\rm iii.}]  If $\Ga$ consists of parabolic translations around a point $p_0 \in \partial_\infty P$, then 
$M$ is a piece of either an ideal bowl soliton  or an ideal translating catenoid.
\end{enumerate}
\end{theorem}
\begin{proof}
Item {\rm (i)} is a direct consequence of Theorem \ref{thm-rev-surf}.
\end{proof}

\section{Translating solitons in $\mathbb{R}\times \mathbb{H}^n$} Now we specialize to the case when $P=\mathbb{H}^n$. 

\subsection{Isometries of $\mathbb{H}^n$}
Recall that $\mathbb{H}^n$ can be realized in Lorentzian space $\mathbb{L}^{n+1}$ as the set
\[
\mathbb{H}^n=\Big\{p=(x_0, x_1, x_2, \ldots, x_n)\in \mathbb{L}^{n+1}: -x_0^2+x_1^2 + \sum_{i=2}^n x_i^2=-1, \, x_0>0\Big\}.
\]
Denote ${\sf e} = (1,0,\ldots, 0)$. The Lorentzian coordinates $(x_0, x_1, \ldots, x_n)$ can be chosen  in such a way that $o= {\sf e}$. Then a point $p$ in the geodesic sphere $B_r (o)\subset \mathbb{H}^n$ can be written in Lorentzian coordinates as
\[
p = \cosh r {\sf e} + \sinh r \omega,
\]
where $\omega$ is a point of the unit sphere in the Euclidean hyperplane orthogonal to ${\sf e}$ in $\mathbb{L}^{n+1}$. 
We then consider a one-parameter family of hyperbolic translations of the form
\[
T_{-r_0} (p) = (x_0\cosh r_0  -x_1\sinh r_0, x_1\cosh r_0  -x_0\sinh r_0, x_2, \ldots, x_n)
\]
with $p=(x_0, x_1, x_2, \ldots, x_n)\in \mathbb{H}^{n}$. In particular
\[
T_{-r_0} (p_0) = o
\]
where $p_0 = (\cosh r_0, \sinh r_0, 0,\ldots, 0)$. It follows that $T_{-r_0} (B_r(o))$ is a geodesic sphere in $\mathbb{H}^n$ centered at $T_{-r_0}(o) = (\cosh r_0, -\sinh r_0, 0, \ldots, 0)$ with radius $r$. In particular, extending $T_{-r_0}$ trivially to an isometry in $\mathbb{R}\times \mathbb{H}^n$ one has
\[
M_{\varepsilon, r_0} := T_{-r_0}(M_\varepsilon)
\]
is a rotationally symmetric translating soliton in $\mathbb{H}^n$ foliated by geodesic spheres centered at $T_{-r_0}(o)$. Those geodesic spheres are given by the intersection of $\mathbb{H}^n$ with timelike hyperplanes in $\mathbb{L}^{n+1}$ of the form
\[
\langle T_{-r_0}(p), T_{-r_0}(o) \rangle =\langle p, o\rangle =-\cosh r,
\]
where $\langle \cdot, \cdot\rangle$ denotes the Lorentzian metric. Those hyperplanes are orthogonal to 
\[
\nu_{r_0} =\frac{1}{\cosh r_0}T_{-r_0}(o)  =  (1, -\tanh r_0, 0,\ldots, 0).
\]
Note that $\nu_{r_0}\to \nu_{-\infty} = (1,-1,0,\ldots, 0)$, a lightlike vector field that determines a family of horospheres $H_a$, $a\in \mathbb{R}$,  given by the intersection  between $\mathbb{H}^n$ and the lightlike hyperplanes
\[
x_0 + x_1 = a.
\]
Then we define the hyperbolic isometry of $\mathbb{H}^n$ that fixes the ideal point represented by the lightlike vector $\nu_{-\infty}$.
\[
T_{-\infty} \big(x_0, x_1, x_2, {\sf x}\big) = \big( x_0 +\alpha^2(x_0 +x_1) +\alpha x_2, -\alpha^2(x_0+x_1) +x_1 -\alpha x_2, \alpha(x_0 + x_1) + x_2, {\sf x}\big),
\]
where ${\sf x} = p-(x_0, x_1, x_2, 0,\ldots, 0)$. Note that $T_{-\infty}(H_a) = H_a$.
It turns out that
\[
M^\infty_\varepsilon = T_{-\infty} (M_\varepsilon).
\]
Now, given the geodesic 
\[
\alpha (\tau) = (\cosh \tau, 0, \sinh\tau, 0,\ldots, 0),
\]
we consider the corresponding family of equidistant hypersurfaces given by the intersection of $\mathbb{H}^n$ with hyperplanes of the form $x_1 = a = \sinh r$, $r\in \mathbb{R}$, that is,
\[
-x_0^2 + x_2^2 + |{\sf x}|^2 = -(1+a^2) = -\cosh^2 r. 
\]
Those equidistant hypersufaces $E_a, \, a\in \mathbb{R}$, are parameterized by
\begin{eqnarray*}
p 
=(\cosh r \cosh \tau,  \sinh r \langle\vartheta, {\sf e}_1\rangle,  \cosh r \sinh \tau, \sinh r (\vartheta-\langle \vartheta, {\sf e}_1\rangle {\sf e}_1))
\end{eqnarray*}
where ${\sf e}_1=(0,1,0,\ldots,0)$ and $\vartheta$ is a point in the unit sphere in the Euclidean $(n-1)$-dimensional plane $x_0=x_2=0$.
It follows that the metric in $\mathbb{H}^n$ is expressed in terms of the coordinates $(r, \tau, \vartheta)$ as
\[
\cosh^2 r \,{\rm d}\tau^2 + {\rm d}r^2 +  \sinh^2 r \,{\rm d}\vartheta^2.
\]
As above, the isometries $T_\varrho$ converge in the limit as $\varrho\to+\infty$  to a parabolic isometry fixing the ideal point $\nu_{+\infty} = (1,1,0,\ldots, 0) \in \partial_\infty \mathbb{H}^n$. Hence, a translating soliton foliated by equidistant hypersurfaces in parallel hyperplanes $P_t$ can be regarded as the limit of the translating solitons foliated by the horospheres given by the intersection of $\mathbb{H}^n$ with the hyperplanes $x_0-x_1=a$, $a\in \mathbb{R}$. In sum, 
\begin{equation}
T_{\varrho}(M_0^\infty)\to M^\infty
\end{equation}
as $\varrho\to +\infty$.

\subsection{Non-existence and uniqueness results}

Now, we are going to consider the $1$-parameter family of translating solitons given by 
\begin{equation}
\mathcal{T}= \{T_{r}(M^\infty_\varepsilon),\, r\in \mathbb{R}\}
\end{equation}

\begin{remark}
Notice that the limit, as $r \to+\infty$, of the surfaces $T_{r}(M^\infty_\varepsilon)$ degenerates into $[t_0,+\infty) \times \partial_\infty \mathbb{H}^n$,
for some $t_0 \in \mathbb{R}$. On the other hand, as $r \to-\infty$, the uniform limit on compact sets of the family
$T_{r}(M^\infty_\varepsilon)$ consists of two copies of the ideal bowl soliton  $M^\infty_0$ described in Theorem {\rm(\ref{thm-rev-surf-ideal})}
\end{remark}

Using the family $\mathcal{T}$, we can prove the following result
\begin{theorem}
There are no complete translating 
solitons with compact {\rm(}probably empty{\rm)} boundary, properly embedded in a solid horocylinder  $\mathcal{O}$ in $\mathbb{R} \times \mathbb{H}^n$.
\end{theorem}
\begin{proof}
Assume there were such a soliton $M$ such that $\partial M$ is compact and $M \subset \mathcal{O}.$
Since there exists $t_0\in \mathbb{R}$ such that 
\[
\partial M \subset [-t_0, t_0] \times \mathbb{H}^n 
\]
we can find $r  \in \mathbb{R}$ such that 
\[
T_{r}(M^\infty_\varepsilon)\subset  [t_0, +\infty)\times  \mathbb{H}^n 
\] 
up to a
vertical translation and $T_{r}(M^\infty_\varepsilon) \cap \mathcal{O}= \emptyset$. Since the limit, as $r \to-\infty $, of $T_{r}(M^\infty_\varepsilon)$ consists
of two copies of the ``ideal" bowl soliton $M^\infty_0$, then we can assert that there is a first point of contact between
some element in the family $\mathcal{T}$, say $T_{r_*}(M^\infty_\varepsilon)$, and the soliton $M$. Due to the assumptions
about the boundary, we have that this first contact between $T_{r_*}(M^\infty_\varepsilon)$ and $M$ occurs at an interior
point of contact. By the maximum principle, we deduce that $M$ is a complete subset of $T_{r_*}(M^\infty_\varepsilon)$ 
with compact boundary, which is absurd because no end of $T_{r_*}(M^\infty_\varepsilon)$ is contained in a horocylinder.
\end{proof}

\begin{theorem}
\label{wing-1}
Let  $M$ be a translating soliton in $\mathbb{R}\times \mathbb{H}^n$ diffeomorphic to a cylinder whose ends 
 $E_-$ and $E_+$, outside the cylinder over a geodesic ball $B_R(o)\subset \mathbb{H}^n$,  are  graphs of smooth functions $u_\pm: \mathbb{H}^n\backslash B_R(o) \to \mathbb{R}$ satisfying
\begin{equation}
\label{lim-u}
\lim_{r=d(o,x) \to\infty} u_\pm(x)=+\infty
\end{equation}
 and 
\begin{equation}
\label{slope-ineq}
\lim_{r \to+\infty} \frac{\partial u_-}{\partial r} < \frac{c}{n-1}< \lim_{r \to+\infty} \frac{\partial u_+}{\partial r} 
\end{equation}
uniformly with respect to $\vartheta \in \partial_\infty \mathbb{H}^n$. Then, $M=M_\varepsilon$, for some $\varepsilon>0$.
\end{theorem}

\noindent \emph{Proof.}  Consider a bowl soliton $M_0$ in $\mathbb{R}\times\mathbb{H}^n$ whose slope at infinity is precisely $\frac{c}{n-1}$, as described in (\ref{deru-asymp}). It follows from (\ref{slope-ineq}) that there is a translated copy $\Phi_{t_0}(M_0)$ of $M_0$ given by the graph of a function $u_0$ such that we have that
\[
u_-(x) < u_0(x) < u_+(x)
\]
for a sufficiently large $R>0$ and every $x\in \mathbb{H}^n\backslash B_R(o)$. Proposition \ref{u-asymp} implies that all the wing-like translating solitons $M_\varepsilon$, $\varepsilon>0,$ have the same asymptotic slope.  Hence
\[
M\cap \Phi_{t_0}(M_\varepsilon) = \emptyset,
\]
for every sufficiently large $\varepsilon>0$.  On the other hand $M_\varepsilon \to M_0$ as $\varepsilon \to 0$. Since $\Phi_{t_0}(M_0)$ intersects $M$ we conclude that
\[
\varepsilon_0 = \inf_{\varepsilon} M\cap \Phi_{t_0}(M_\varepsilon) \neq \emptyset >0.
\]
Then, $M_{\varepsilon_0}$ is tangent to $M$. A direct application of the maximum principle implies that $M=M_{\varepsilon_0}$ what contradicts (\ref{slope-ineq}). \hfill $\square$

\vspace{3mm}

Reasoning as above, we get the following consequence of the proof of Theorem \ref{wing-1}. Notice that now the proof is even simpler since one of the ends has height bounded from above. 

\begin{theorem}
\label{wing-2}
There are no  translating solitons $M$ in $\mathbb{R} \times \mathbb{H}^n$ diffeomorphic to a cylinder  such that the height function $t|_M$ is bounded above on one of its ends and the other end is, outside the cylinder over a geodesic ball $B_R(o)\subset \mathbb{H}^n$, given by the graph of a smooth function $u: \mathbb{H}^n\backslash B_R(o) \to \mathbb{R}$ satisfying
\[
\lim_{r=d(o,x) \to\infty} u(x)=+\infty \quad {\rm and}\quad \lim_{r \to+\infty} \frac{\partial u}{\partial r} > \frac{c}{n-1}\cdot
\]
\end{theorem}

%


\section{Applications of the maximum principle}

The following result is a analog of Theorem 1 in \cite{ABD} for the case of translating solitons. Then can be understood as a weighted counterpart of a mean curvature estimate.

\begin{theorem}Suppose that the radial sectional curvatures of $P^n$ along geodesics issuing from a given point $o\in P$ satisfy
\begin{equation}
K_{{\rm rad}} \le -\frac{\xi''}{\xi}\cdot
\end{equation}
Then there are no translating solitons properly immersed in a cylinder $\mathcal{C}$ over a geodesic ball in $P$ centered at $o$.
\end{theorem}

\noindent \emph{Proof.} Let $d(x) = {\rm dist}(\psi(x), \mathcal{C})$ and set $\chi = \bar \chi\circ d$. We have $\nabla\chi = \bar\chi' (d) \nabla d$ and
\[
\Delta \chi = \bar\chi '' |\nabla d|^2 + \bar\chi' \Delta d 
\]
and
\begin{equation*}
\Delta_{-c\eta} \chi = \Delta \chi + c\langle \nabla\eta, \nabla\chi\rangle  = \Delta \chi + c\bar\chi'\langle X, \nabla d\rangle 
\end{equation*}
what implies that
\begin{equation}
\Delta_{-c\eta} \chi = \bar\chi'' |\nabla d|^2+ \bar\chi' (\Delta d+ c\langle X, \nabla d\rangle).
\end{equation}
Denote by $II_{\mathcal{C}}$ is the second fundamental form $\mathcal{C}$. It is well known that 
\begin{equation}
\Delta d = -\sum_{i=1}^m \langle II_{\mathcal{C}} ({\sf e}_i^\top, {\sf e}_i^\top), \bar\nabla d\rangle + \langle {\bf H}, \bar\nabla d^\perp\rangle,
\end{equation}
where $\top$ denotes the tangential projection onto the level sets of the distance ${\rm dist}(\cdot, \mathcal{C})$  and $\perp$ denotes the projection onto the normal bundle of $\psi$. It follows that
\[
\Delta_{-c\eta} d =  -\sum_{i=1}^m\langle  II_{\mathcal{C}} ({\sf e}_i^\top, {\sf e}_i^\top) -c X, \bar\nabla d\rangle + \langle {\bf H}-cX^\perp, \bar\nabla d\rangle
\]
what yields
\begin{equation*}
\Delta_{-c\eta}\chi = \bar\chi''|\nabla d|^2 + \bar\chi' \Big(-\sum_{i=1}^m\langle  II_{\mathcal{C}} ({\sf e}_i^\top, {\sf e}_i^\top) -c X, \bar\nabla d\rangle + \langle {\bf H}-cX^\perp, \bar\nabla d\rangle\Big).
\end{equation*}
In the particular case when $\mathcal{C} = \{o\}\in P$ (or equivalently, $\mathcal{C}$ is a cylinder over a geodesic ball in $P$ centered at $o$) we have
\[
\langle X, \bar\nabla d\rangle =0
\]
and 
\[
II_{\mathcal{C}} ({\sf e}_i^\top, {\sf e}_i^\top)  = II_{\mathcal{C}}({\sf e}_i - \langle {\sf e_i}, X\rangle X - \langle {\sf e}_i, \bar\nabla d\rangle\bar\nabla d, {\sf e}_i - \langle {\sf e_i}, X\rangle X - \langle {\sf e}_i, \bar\nabla d\rangle\bar\nabla d).
\]
The Hessian comparison theorem implies that
\[
II_{\mathcal{C}} ({\sf e}_i^\top, {\sf e}_i^\top)  \le -\frac{\xi'(d)}{\xi(d)} \langle {\sf e}_i - \langle {\sf e_i}, X\rangle X - \langle {\sf e}_i, \bar\nabla d\rangle\bar\nabla d, {\sf e}_i - \langle {\sf e_i}, X\rangle X - \langle {\sf e}_i, \bar\nabla d\rangle\bar\nabla d\rangle 
\]
under the assumption that the radial sectional curvatures in $P$ along geodesics issuing from $o\in P$ satisfies
\[
K_{{\rm rad}} (d) \le -\frac{\xi''(d)}{\xi(d)}\cdot
\]
Taking traces
\[
\sum_{i=1}^m II_{\mathcal{C}} ({\sf e}_i^\top, {\sf e}_i^\top)  \le -\frac{\xi'(d)}{\xi(d)} \big(m-2+|X^\perp|^2+|\bar\nabla d^\perp|^2\big) 
\]
where as above $\perp$ indicates the normal projection onto $\psi(M)$. In this particular case we have (if $\bar \xi'>0$)
\begin{eqnarray*}
\Delta_{-c\eta}\chi \ge  \bar\chi''|\nabla d|^2 +\big(m-2+|X^\perp|^2+|\bar\nabla d^\perp|^2\big)\chi' \frac{\xi'}{\xi}+\langle {\bf H}-cX^\perp, \bar\nabla d\rangle\chi'
\end{eqnarray*}
Now having fixed
\[
\chi(d) = \int^d_0 \xi (\tau)\, {\rm d}\tau
\]
one obtains
\begin{eqnarray*}
& & \Delta_{-c\eta}\chi \ge  \xi' (1-|\bar\nabla d^\perp|^2) + \big(m-2+|X^\perp|^2+|\bar\nabla d^\perp|^2\big)\xi' +\langle {\bf H}-cX^\perp, \bar\nabla d\rangle\xi\ge \\
& &\,\, =\big(m-1+|X^\perp|^2\big) \xi' +\langle {\bf H}-cX^\perp, \bar\nabla d\rangle\xi \ge \xi\bigg((m-1)\frac{\xi'(d)}{\xi(d)} - |{\bf H}-cX^\perp|\bigg)
\end{eqnarray*}
Suppose that $\psi(M)$ is contained in the cylinder over a geodesic ball of radius $d^*$ in $P$ centered at $o$. Hence 
\[
\chi \le  \bar \chi (d^*) <+\infty
\]
If
\begin{equation}
\sup_M |{\bf H}-cX^\perp| < (m-1)\frac{\xi'(d^*)}{\xi(d^*)}
\end{equation}
then
\begin{equation}
\Delta_{-c\eta}\chi(x) >0
\end{equation}
whenever 
\[
\chi(x) > \chi(d^*)-\varepsilon
\]
for some sufficiently small $\varepsilon>0$. This contradicts the weak maximum principle. The validity of this principle for the operator $\Delta_{-c\eta}$ follows from the fact that the $\psi$ is a proper immersion. Then the function $\eta = \pi\circ\psi$ satisfies
\[
\eta(x)  \to  \infty \quad \mbox{as}\quad x\to \infty
\]
and
\begin{eqnarray*}
\Delta_{-c\eta}\eta =c
\end{eqnarray*}
This is enough to guarantee the validity of the weak maximum principle.
\hfill $\square$
%
%
%

\begin{theorem} \label{thm-ultimo}
Let $P$ be a $n$-dimensional complete Riemannian manifold endowed with a rotationally invariant metric $g_0$ whose sectional curvatures are negative. There are no complete translating solitons  contained in a horocylinder of $\R \times P$.
\end{theorem}

\begin{proof}
We proceed by contradiction. Let $\mathcal{O}$ be a horocylinder in $\R \times P$ and assume that there is a complete soliton $M \subset \mathcal{O}$. Assume that the ideal points of $\mathcal{O}$ lie on the line $\R \times\{x_\infty\}$, where $x_\infty \in \partial_\infty P$. 

Then, we consider the family of solitons  $\{M_\varepsilon ^\infty, \; \varepsilon \in [0,+\infty)\}$ given by Theorem \ref{thm-rev-surf-ideal}. Fix $\varepsilon>0$ large enough so that the horocylinder $\mathcal{O}$  is contained in the mean convex region of $(\R \times P)-M_\varepsilon^\infty,$ then $M_\varepsilon^\infty \cap M= \varnothing.$ The limit as $ \varepsilon \to 0$ of the hypersurfaces in this family  consists of two copies of $M_0^\infty$ which intersect $M$. Then there is a first point of contact between an element $M^\infty_{\varepsilon_0}$ and $M$. By the asymptotic behaviour of $M^\infty_{\varepsilon_0}$, this contact cannot occur at infinity. So, there must be an interior point of contact and using the maximum principle, we deduce that $M^\infty_{\varepsilon_0}=M$, which is a contradiction since $M^\infty_{\varepsilon_0}$ is not contained in $\mathcal{O}.$
\end{proof}
As a corollary we get that there are no complete cylindrically bounded translating solitons.
\begin{corollary}
Let $P$ be a $n$-dimensional complete Riemannian manifold endowed with a rotationally invariant metric $g_0$ whose sectional curvatures are negative. There are no complete translating solitons  contained in a cylinder  of $\R \times P$.
\end{corollary}

\begin{remark}
Notice that the proof of Theorem \ref{thm-ultimo} still works if $M$ has compact boundary and the Euclidean coordinate is unbounded on $M$.
\end{remark}

\begin{corollary}
Let $P$ be a $n$-dimensional complete Riemannian manifold endowed with a rotationally invariant metric $g_0$ whose sectional curvatures are negative. There are no complete translating graphs over a domain contained in a horodisc in $P$.
\end{corollary}

%
%
%
%
%

\end{document}